\documentclass[a4paper, 12pt, leqno]{article}
\usepackage[utf8]{inputenc}
\usepackage[T1]{fontenc}

\usepackage{amsmath}
\usepackage{amssymb}
\usepackage{amsthm}
\usepackage{hyperref}

\numberwithin{equation}{section}

\numberwithin{equation}{section}

\newtheorem{theorem} {Theorem} [section]
\newtheorem{lemma} [theorem]  {Lemma}
\newtheorem{corollary} [theorem] {Corollary}

\theoremstyle{definition}
\newtheorem{example} [theorem] {Example}

\title{{\Large \bf 
Polynomials with exponents in compact convex sets and 
associated weighted extremal functions - \\
Approximations and regularity}}
\author{Bergur Snorrason}
\date{}
\begin{document}

\maketitle

\begin{abstract}
\noindent
We study various regularization operators on plurisubharmonic functions that
    preserve Lelong classes with growth given by certain compact convex sets.
The purpose is to determine when the weighted Siciak-Zakharyuta functions associated
    with these Lelong classes are lower semicontinuous.
The operators are given by integral, infimal, and supremal convolutions.
Continuity properties of the logarithmic supporting function are studied
    and a precise description is given of when it is uniformly continuous.
This gives a contradiction to published results about the Hölder continuity of these Siciak-Zakharyuta functions.

\end{abstract}
\section{Introduction}
\label{sec:1}
A useful tool in pluripotential theory is regularization by convolution.
If $u \in \mathcal{PSH}(\Omega)$, for some open $\Omega \subset \mathbb{C}^n$,
    then $u * \chi_\delta \in \mathcal{PSH} \cap \mathcal{C}^\infty(\Omega_\delta)$, for $\delta > 0$,
    where we define
\begin{equation*}
    u * \chi_\delta(z)
    =
    \int_{\mathbb{C}^n}
        u(z - w) \chi_\delta(w)\,d\lambda(w),
    \quad
    z \in \Omega_\delta.
\end{equation*}
Here, $\Omega_\delta = \{z \in \Omega \,;\, B(z, \delta) \subset \Omega\}$,
        where $B(z, \delta)$ is the open euclidean ball with center $z$ and radius $\delta$,
    $\chi_\delta$ is a standard smoothing kernel in $\mathbb{C}^n$ with support in the closed ball
    $\overline{B}(0, \delta)$, and
    $\mathcal{C}^\infty(\Omega_\delta)$ is the family of smooth function on $\Omega_\delta$.
We also have that $u * \chi_\delta \searrow u$, as $\delta \searrow 0$.
In the specific setting of $\Omega = \mathbb{C}^n$, we have $\Omega_\delta = \mathbb{C}^n$.
So, every function in $\mathcal{PSH}(\mathbb{C}^n)$ can approximated by a decreasing sequence in
    $\mathcal{PSH} \cap \mathcal{C}^\infty(\mathbb{C}^n)$.
For details on smoothing by convolution, see Klimek \cite[Thm.~2.9.2]{Kli:1991}.

In some cases, it is necessary to choose a method for regularizing plurisubharmonic functions
    that preserves a certain subclass of $\mathcal{PSH}(\mathbb{C}^n)$.
One such case is the Lelong class $\mathcal{L}(\mathbb{C}^n)$.
We say that $u \in \mathcal{PSH}(\mathbb{C}^n)$ belongs to the Lelong class $\mathcal{L}(\mathbb{C}^n)$
    if $u \leq \log^+\| \cdot \|_\infty + c_u$ for some constant $c_u$,
    where $\| \cdot \|_\infty$ is the supremum norm in $\mathbb{C}^n$ and $\log^+ = \max\{0, \log\}$.
A powerful tool in the study of pluripotential theory is the Siciak-Zakharyuta function of $E \subset \mathbb{C}^n$,
    defined by
\begin{equation*}
    V_E(z)
    =
    \sup\{u(z) \,;\, u \in \mathcal{L}(\mathbb{C}^n), u|_E \leq 0\},
    \quad
    z \in \mathbb{C}^n.
\end{equation*}
One can show that smoothing with a standard smoothing kernel preserves the Lelong class.
Namely, if $u \in \mathcal{L}(\mathbb{C}^n)$ then $u * \chi_\delta \in \mathcal{L}(\mathbb{C}^n)$.
This is used to show that $V_K$ is lower semicontinuous,
  for compact non-pluripolar $K \subset \mathbb{C}^n$.
See Klimek \cite[Section~5.1]{Kli:1991}.

In the recent study of pluripotential theory related to convex sets,
    the Lelong classes are generalized by describing the growth more freely.
We fix a compact convex set $0 \in S \subset \mathbb{R}^n_+$ and recall that the \emph{supporting function of $S$},
    given by $\varphi_S \colon \mathbb{R}^n \rightarrow \mathbb{R}$, $\varphi_S(\xi) = \sup_{x \in S} \langle x, \xi \rangle$,
    is positively homogeneous and convex,
    which is equivalent to $\varphi_S(t\xi) = t \varphi_S(\xi)$ and
    $\varphi_S(\xi + \eta) \leq \varphi_S(\xi) + \varphi_S(\eta)$,
    for $t \in \mathbb{R}_+$ and $\xi, \eta \in \mathbb{R}^n$.
We define the \emph{logarithmic supporting function of $S$} by
\begin{equation*}
    H_S(z)
    =
    \left \{
    \begin{array}{l l}
        \varphi_S(\log |z_1|, \dots, \log |z_n|), & z \in \mathbb{C}^{*n},\\
        \varlimsup\limits_{\mathbb{C}^{*n} \ni w \rightarrow z} H_S(w), & z \in \mathbb{C}^n \setminus \mathbb{C}^{*n},
    \end{array}
    \right.
\end{equation*}
where $\mathbb{C}^{*n} = (\mathbb{C}^*)^n$ and $\mathbb{C}^* = \mathbb{C} \setminus \{0\}$.
This function is continuous and plurisubharmonic on $\mathbb{C}^n$ and
    its behavior on the coordinate hyperplanes $\mathbb{C}^n \setminus \mathbb{C}^{*n}$
    can be described using logarithmic supporting functions of sets of lower dimensions.
See \cite[Props.~3.3 and 3.4]{MagSigSigSno:2023}.
We define the \emph{Lelong class} $\mathcal{L}^S(\mathbb{C}^n)$ \emph{given by $S$}
    as the set of all $u \in \mathcal{PSH}(\mathbb{C}^n)$ satisfying
    $u \leq H_S + c_u$, for some constant $c_u$.
Similarly, we define $\mathcal{L}^S_+(\mathbb{C}^n)$ as the set of all functions $u \in \mathcal{L}^S(\mathbb{C}^n)$
    satisfying $H_S + c_u \leq u$, for some constant $c_u$.
This leads to our definition of the
    \emph{Siciak-Zakharyuta function of $E \subset \mathbb{C}^n$ and $S$ with weight
    $q \colon E \rightarrow \mathbb{R} \cup \{+\infty\}$},
\begin{equation*}
    V^S_{E, q}(z)
    =
    \sup\{u(z) \,;\, u \in \mathcal{L}^S(\mathbb{C}^n), u|_E \leq q\},
    \quad
    z \in \mathbb{C}^n.
\end{equation*}
We write $V^S_K = V^S_{K, 0}$, $V_{K, q} = V^\Sigma_{K, q}$, and $V_K = V^\Sigma_{K, 0}$,
    where $\Sigma$ is the standard simplex in $\mathbb{R}^n$
    given by $\Sigma = \operatorname{ch}\{0, e_1, \dots, e_n\}$, where
    $\operatorname{ch} A$ denotes the closed convex hull of the set $A$.
This is justified by the fact that
$
  \mathcal{L}^\Sigma(\mathbb{C}^n)
=
  \mathcal{L}(\mathbb{C}^n).
$

A weight $q \colon E \rightarrow \mathbb{R} \cup \{+\infty\}$ is called \emph{admissible on $E$} if
    $q$ is lower semicontinuous,
    the set $\{z \in E \,;\, q(z) < +\infty\}$ is non-pluripolar, and
    $\lim_{E \ni z, |z| \rightarrow \infty} (H_S(z) - q(z)) = -\infty$, if $E$ is unbounded.
Note that, the existence of an admissible weight on $E$ implies that $E$ is non-pluripolar.

Fundamental results of these Siciak-Zakharyuta function are developed in \cite{MagSigSigSno:2023}.
There, it is proven that $V^S_{K, q}$ is lower semicontinuous on $\mathbb{C}^{*n}$,
    but it is not proven that it is lower semicontinuous on all of $\mathbb{C}^n$.
The obstruction to proving lower semicontinuity on $\mathbb{C}^n$ is a suitable method of approximating functions in
    $\mathcal{L}^S(\mathbb{C}^n)$
    by functions in $\mathcal{L}^S \cap \mathcal{C}(\mathbb{C}^n)$,
    where $\mathcal{C}(\mathbb{C}^n)$ denotes all continuous functions on $\mathbb{C}^n$,
    since smoothing by convolution with a standard smoothing kernel does not always preserve
    $\mathcal{L}^S(\mathbb{C}^n)$.
In fact, $\mathcal{L}^S(\mathbb{C}^n)$ is preserved by the standard convolution operator if and only if $S$ is a
    \emph{lower set},
    that is when $S$ is such that for all $x \in S$ we have that $[0, x_1] \times \dots \times [0, x_n] \subset S$.
See \cite[Thm.~5.8]{MagSigSigSno:2023}.

In Section \ref{sec:2}, we demonstrate the connection between the lower semicontinuity of $V^S_{K, q}$
    and regularization in $\mathcal{L}^S(\mathbb{C}^n)$.
We also show that if $S$ contains a neighborhood of $0$ in $\mathbb{R}^n_+$ then
    the standard convolution operator,
    along with gluing,
    suffices to show that $V^S_{K, q}$ is lower semicontinuous.
Namely, we prove the following.
\begin{theorem}
\label{thm:1.6}
Let
    $S \subset \mathbb{R}^n_+$ be a convex compact neighborhood of $0$ in $\mathbb{R}^n_+$,
    $K \subset \mathbb{C}^n$ be compact, and
    $q \colon K \rightarrow \mathbb{R} \cup \{+\infty\}$ an admissible weight on $K$.
Then $V^S_{K, q}$ is lower semicontinuous.
\end{theorem}

This is significantly stronger than the case when $S$ is a lower set.
For example, if $S_1, S_2, \dots$ are lower sets,
    then so is $S = \cap_{j = 1}^\infty S_j$.
However,
    for an arbitrary convex compact $S \subset \mathbb{R}^n_+$, with $0 \in S$,
    we can define $S_j = \operatorname{ch}\big ((1/j) \Sigma \cup S \big)$.
Then $S = \cap_{j = 1}^\infty S_j$, so
  $S$ is the intersection of a decreasing sequence of neighborhoods of $0$ in $\mathbb{R}^n_+$.
This idea, along with \cite[Prop.~4.8(i)]{MagSigSigSno:2023},
    can sometimes be used to strengthen results,
  as is done in \cite[the proof of Thm.~1.1]{Sno:2024a}.

An important consequence of Theorem \ref{thm:1.6} is sharpening the Siciak-Zakharyuta theorem proved in
    \cite[Thm.~1.1]{MagSigSig:2023}.

In Sections \ref{sec:3}-\ref{sec:5} we study regularization operators other than standard convolution operator.
In Section \ref{sec:3} we consider a generalization of an operator of Siciak \cite[Prop.~1.3]{Sic:1981}.
\begin{theorem}
\label{thm:1.1}
Let
    $S \subset \mathbb{R}_+^n$ be a lower set,
    $u \in \mathcal{L}^S(\mathbb{C}^n)$ be bounded below,
    $\mu$ be a distance function on $\mathbb{C}^n$,
    $u \leq H_S + c_u$ for some constant $c_u$, and
    $\delta \in ]0, \sigma_S^{-1}r_\mu e^{c_u}[$,
        where $\sigma_S = \varphi_S(1, \dots, 1)$ and $r_\mu = \inf_{|z| = 1} \mu(z) > 0$.
Then
\begin{equation}
\label{eq:1.1}
    R^a_{\mu, \delta} u(z)
    =
    -\log \inf_{w \in \mathbb{C}^n}
    \{
        e^{-u(w)} + \delta^{-1} \mu(z - w)
    \},
    \quad
    z \in \mathbb{C}^n,
\end{equation}
is in $\mathcal{L}^S(\mathbb{C}^n) \cap \mathcal{C}(\mathbb{C}^n)$ and $R^a_{\mu, \delta} u \searrow u$,
    as $\delta \searrow 0$.
\end{theorem}

With $\mu = | \cdot |$ and $S = \Sigma$, Siciak proved that
    $R^a_{\mu, \delta} u \in \mathcal{L}(\mathbb{C}^n) \cap \mathcal{C}(\mathbb{C}^n)$
    and that $R^a_{\mu, \delta} u \searrow u$, as $\delta \searrow 0$.
In \cite{BayHusLevPer:2020} this process is referred to as \emph{Ferrier approximation},
    since Siciak developed it using a result from Ferrier \cite[Lemma 2, p.~48]{Fer:1973}.
Here it is attributed to Siciak, since Ferrier was not concerned with approximations.
It is claimed in \cite[Prop.~3.1]{BayHusLevPer:2020} that if
    $\mu = | \cdot |$, $u \in \mathcal{L}^S(\mathbb{C}^n)$, and
    $S$ is such that it contains a neighborhood of $0$ in $\mathbb{R}^n_+$,
    then $R^a_{\mu, \delta} u \in \mathcal{L}^S(\mathbb{C}^n) \cap \mathcal{C}(\mathbb{C}^n)$.
Example~\ref{ex:3.1} shows that this claim is false.

The infimum in the definition of the operator $R^a_{\mu, \delta}$ is an
    \emph{infimal convolution} of the functions $e^{-u}$ and
    $\delta^{-1}| \cdot |$.
Details on infimal convolutions and their applications in convexity theory can be found in Rockafellar~\cite{Rockafellar:1970}.
Applications in complex analysis can be found in Kiselman~\cite{Kis:1993} and Halvarsson~\cite{Hal:1994, Hal:1996a, Hal:1996b}.

In Section \ref{sec:4} we consider the following \emph{supremal convolution}.
\begin{theorem}
\label{thm:1.3}
Let
    $0 \in S \subset \mathbb{R}^n_+$ be compact convex,
    $1_n = (1, \dots, 1) \in \mathbb{N}^n$,
    $\sigma_S = \varphi_S(1_n)$,
    $\delta \in ]0, \sigma_S^{-1}[$,
    $u \in \mathcal{L}^S(\mathbb{C}^n)$ be locally bounded,
    and set
\begin{equation*}
    R^b_\delta u(z)
    =
    \sup_{w \in \mathbb{C}^n}
    \big \{
        u(Zw) - \delta^{-1}\log(\|w - 1_n\|_\infty + 1)
    \big \},
    \quad
    z \in \mathbb{C}^n,
\end{equation*}
where $Zw = (z_1w_1, \dots, z_nw_n)$.
Then
    $R^b_\delta u \in \mathcal{L}^S(\mathbb{C}^n)$,
    $R^b_\delta u$ is continuous on $\mathbb{C}^{*n}$, and
    $R^b_\delta u \searrow u$, as $\delta \searrow 0$.
\end{theorem}
The meaning of $Zw$ is justified in Section \ref{sec:4}.
We note that $\sigma_S$ may be $0$.
This only happens when $S = \{0\}$, which is a case of no interest, since then $\mathcal{L}^S(\mathbb{C}^n)$
    would only contain the constant functions.

In Section \ref{sec:5} we consider two related operators.
The first one was developed in \cite[Section 5]{MagSigSigSno:2023}.
\begin{theorem}
\label{thm:1.4}
Let
    $0 \in S \subset \mathbb{R}^n_+$ be compact convex,
    $\delta \in ]0, 1[$, 
    $u \in \mathcal{L}^S(\mathbb{C}^n)$,
    and set
\begin{equation*}
    R^c_\delta u(z)
    =
    \int_{\mathbb{C}^n} u(Zw) \psi_\delta(w)\,d\lambda(w),
    \quad
    z \in \mathbb{C}^n,
\end{equation*}
where $\psi_\delta(w) = \delta^{-2n} \psi\big((w - 1_n)/\delta\big)$
    and $\psi$ is a smoothing kernel that is rotationally invariant in each variable.
Then
    $R^c_\delta u \in \mathcal{L}^S(\mathbb{C}^n)$,
    $R^c_\delta u$ is smooth on $\mathbb{C}^{*n}$, and
    $R^c_\delta u \searrow u$, as $\delta \searrow 0$.
If $u$ is locally bounded below and $\lim_{z \rightarrow a} u_z = u_a$, in $L^1_{\text{loc}}(\mathbb{C}^n)$,
    for all $a \in \mathbb{C}^n$, where $u_z$ is given by $u_z(w) = u(Zw)$, for $w \in \mathbb{C}^n$,
    then $R^c_\delta u \in \mathcal{C}(\mathbb{C}^n)$.
\end{theorem}
The assumption that $u$ is locally bounded below ensure that $R^c_\delta u > -\infty$.
Taking $u(z) = \log |z_1|$, we have $R^c_\delta u(z) = -\infty$, for $z \in \{0\} \times \mathbb{C}^{n - 1}$.
So, $R^c_\delta u$ fails to be continuous on $\mathbb{C}^n$.

A possible approach to showing continuity on all of $\mathbb{C}^n$ is to regularize $e^u$ instead of $u$.
The second operator in Section \ref{sec:5} is given by $R^d_\delta u = \log R^c_\delta e^u$.
It must be shown that this function is plurisubharmonic.
\begin{theorem}
\label{thm:1.5}
Let
    $0 \in S \subset \mathbb{R}^n_+$ be compact convex,
    $\delta \in ]0, 1[$, 
    $u \in \mathcal{L}^S(\mathbb{C}^n)$,
    and set
\begin{equation*}
    R^d_\delta u(z)
    =
    \log \int_{\mathbb{C}^n} e^{u(Zw)} \psi_\delta(w)\,d\lambda(w),
    \quad
    z \in \mathbb{C}^n.
\end{equation*}
Then
    $R^d_\delta u \in \mathcal{L}^S(\mathbb{C}^n)$,
    $R^d_\delta u$ is smooth on $\mathbb{C}^{*n}$, and
    $R^d_\delta u \searrow u$, as $\delta \searrow 0$.
\end{theorem}

In a recent paper of Perera \cite{Per:2024}, a classical result on the regularity of $V_K$ is considered,
    and attempts are made to generalize it to the setting of $V^S_{K, q}$.
Perera \cite{Per:2024} claims to give sufficient conditions for $V^S_{K, q}$ to be Hölder continuous on $\mathbb{C}^n$,
    for convex bodies $S$.
In Section, \ref{sec:6} we show that this can not hold without additional restrictions on $S$.
The main result, Corollary \ref{cor:5.2},
    states that if $S$ is not a lower set then $\mathcal{L}^S_+(\mathbb{C}^n)$
    contains no uniformly continuous functions.
Since $V^{S*}_{K, q} \in \mathcal{L}^S_+(\mathbb{C}^n)$, by \cite[Prop.~4.5]{MagSigSigSno:2023},
    we have that
    $V^S_{K, q}$ can not be Hölder continuous if $S$ is not a lower set.

\subsection*{Acknowledgment}  
The results of this paper are a part of a research project, 
\emph{Holomorphic Approximations and Pluripotential Theory},
with  project grant 
no.~207236-051 supported by the Icelandic Research Fund.
I would like to thank the Fund for its support and the Mathematics Division,
    Science Institute,
    University of Iceland,
    for hosting the project.
I thank my supervisors Benedikt Steinar Magnússon and Ragnar Sigurðsson
    for their guidance and careful reading of the paper.

\section{Regularity of the Siciak-Zakharyuta function}
\label{sec:2}
The main motivation for studying the regularization operators in Sections \ref{sec:3}-\ref{sec:5} is proving that
    $V^S_{K, q}$ is lower semicontinuous on $\mathbb{C}^n$,
    not just $\mathbb{C}^{*n}$, for compact $K \subset \mathbb{C}^n$ and an admissible weight $q$ on $K$.
To prove lower semicontinuity using regularizations, we use the following variant of Dini's theorem.
\begin{theorem}
\label{thm:2.1}
\looseness=-1
Let
    $f_1 \geq f_2 \geq \dots$ be a sequence of upper semicontinuous functions on a compact metric space $X$,
        with pointwise limit $f$, and
    $g$ be a lower semicontinuous function on $X$ such that $f < g$.
Then there exists $j_0 \in \mathbb{N}$ such that $f_j < g$, for all $j > j_0$.
\end{theorem}

\begin{proof}
For every $x \in X$, we can take $j_x \in \mathbb{N}$ such that $f_{j_x}(x) - g(x) < 0$.
Since $f_{j_x}$ is upper semicontinuous and $g$ is lower semicontinuous,
    we can take an open neighborhood $U_x$ of $x$ such that $f_{j_x}(y) - g(y) < 0$, for $y \in U_x$.
Since $X$ is compact, we can take $x_1, \dots, x_\ell \in X$ such that
    $U_{x_1}, \dots, U_{x_\ell}$ is a cover of $X$.
Let $j_0 = \max\{j_{x_1}, \dots, j_{x_\ell}\}$.
Then,
    since $(f_j)_{j \in \mathbb{N}}$ is decreasing,
    we have $f_j < g$, for $j > j_0$.
\end{proof}

\looseness=-1
To see how regularization is used to show that $V^S_{K, q}$ is lower semicontinuous,
    we fix an open set $\Omega \subset \mathbb{C}^n$ and assume that
    for all $u \in \mathcal{L}^S(\mathbb{C}^n)$ there exists a decreasing sequence $(u_j)_{j \in \mathbb{N}}$ 
    in $\mathcal{L}^S(\mathbb{C}^n)$ such that
    $u_j|_\Omega$ is continuous and $u_j \searrow u$.
If we take $\varepsilon > 0$ and assume that $u \leq q$ on $K$ then, by Dini's theorem,
    we can find $j \in \mathbb{N}$, such that
    $u_j - \varepsilon \leq q$ on $K$.
Hence, $V^S_{K, q}$ is given as the supremum over a family of functions continuous on $\Omega$,
    and is therefore lower semicontinuous on $\Omega$.
See Klimek \cite[Lemma~2.3.2]{Kli:1991}.

The regularizations considered in this paper work for $\Omega = \mathbb{C}^{*n}$.
Whether these, or any other, regularizations work on all of $\mathbb{C}^n$ remains an open question.

\looseness=-1
We now turn to the main result in this section, Theorem \ref{thm:1.6}.
To prove it we will regularize using standard convolution,
    as discussed at the start on Section \ref{sec:1}.
This regularization will generally not have that right growth,
    but we will fix this using gluing.

\begin{proof}[Proof of Theorem \ref{thm:1.6}]
Since $q$ is admissible on $K$, it is bounded below.
So, we may assume that $q > 0$, and therefore, that $V^S_{K, q} > 0$,
    since $V^S_{K, q + c} = V^S_{K, q} + c$, for every constant $c$.
There also exists a constant $a > 0$, such that $a \Sigma \subset S$,
    since $S$ contains a neighborhood of $0$ in $\mathbb{R}^n_+$.
This implies that $H_S \geq H_{a\Sigma} \geq a \log^+\| \cdot \|_\infty$.

Fix $w \in \mathbb{C}^n$ and $\varepsilon > 0$,
    take $u \in \mathcal{L}^S(\mathbb{C}^n)$ such that $u|_K \leq q$, and
    let $c_u$ such that $u \leq H_S + c_u$
    and $t \in ]0, 1[$ such that $tu(w) + \varepsilon/2 \geq u(w)$.
By \cite[Prop.~4.5]{MagSigSigSno:2023}, there exists a constant $C > 0$ such that $H_S - C \leq V^S_{K, q}$.
But,
\begin{equation*}
\begin{aligned}
    H_S - C - tu
    &\geq
    H_S - tH_S - C - tc_u
\\
    &\geq
    (1 - t)a \log^+\| \cdot \|_\infty - C - tc_u,
\end{aligned}
\end{equation*}
    so $tu < H_S - C$ on $\partial (r\mathbb{D}^n)$, for $r > R_0 = \max\{1, e^{(C + tc_u)/((1 - t)a)}\}$.
We fix $R > R_0$ such that $w \in R\mathbb{D}^n$.
By Theorem \ref{thm:2.1}, we can choose $\delta > 0$ such that

\smallskip
\begin{tabular}{rl}
    (i)   & $(tu) * \chi_\delta (w) + \varepsilon > u(w)$,
\\
    (ii)  & $(tu) * \chi_\delta < H_S - C$ on $\partial(R \mathbb{D}^n)$, and
\\
    (iii) & $(tu) * \chi_\delta < q$ on $K$,
\end{tabular}

\smallskip
\noindent
since $q > 0$, and therefore $tu < q$.
The function
\begin{equation*}
    u_{\varepsilon, w}(z)
    =
    \bigg \{
    \begin{array}{ll}
        \max
        \big \{
          H_S(z) - C, (tu) * \chi_\delta(z)
        \big \},
        &
        z \in R \mathbb{D}^n,
\\
        H_S(z) - C,
        &
        z \in \mathbb{C}^n \setminus R \mathbb{D}^n,
    \end{array}
\end{equation*}
is plurisubharmonic, by Klimek \cite[Cor.~2.9.15]{Kli:1991}.
So $u_{\varepsilon, w} \in \mathcal{L}^S(\mathbb{C}^n)$.
We also have that $H_S - C \leq V^S_{K, q} \leq q$ on $K$,
    so $u_{\varepsilon, w} \leq q$ on $K$.
Since $w \in R\mathbb{D}^n$, we have
\begin{equation*}
    u_{\varepsilon, w}(w) \geq (tu) * \chi_\delta(w) > u(w) - \varepsilon.
\end{equation*}
We also have that $u_{\varepsilon, w}$ is continuous on $\mathbb{C}^n$,
    since $H_S - C$ and $(tu) * \chi_\delta$ both are.

For every $w \in \mathbb{C}^n$ and $\varepsilon > 0$, we can take $u \in \mathcal{L}^S(\mathbb{C}^n)$,
    such that $u|_K \leq q$ and $u(w) > V^S_{K, q}(w) - \varepsilon$.
Since $u_{\varepsilon, w}(w) > u(w) - \varepsilon$, we have that
\begin{equation*}
    V^S_{K, q}(z)
    =
    \sup
    \{
        u_{\varepsilon, w}(z)
        \,;\,
        u \in \mathcal{L}^S(\mathbb{C}^n),\ 
        \varepsilon > 0,\ 
        w \in \mathbb{C}^n
    \},
    \quad
    z \in \mathbb{C}^n.
\end{equation*}
So, $V^S_{K, q}$ is given as the supremum over a family of continuous functions on $\mathbb{C}^n$, and
  is therefore lower semicontinuous $\mathbb{C}^n$.
\end{proof}

\section{Infimal convolutions}
\label{sec:3}
Recall that Siciak's infimal convolution was given by
\begin{equation*}
    R^a_{\mu, \delta} u(z)
    =
    -\log \inf_{w \in \mathbb{C}^n}
    \{
        e^{-u(w)} + \delta^{-1}\mu(z - w)
    \},
    \quad
    z \in \mathbb{C}^n,
\end{equation*}
with $\mu = | \cdot |$.
In Ferrier \cite{Fer:1973}, it is proven that $R^a_{\mu, \delta} u$ is plurisubharmonic and continuous when
    $u$ is plurisubharmonic, $\delta = 1$, and $\mu = | \cdot |$.
Generalizing for $\delta > 0$ is not significant.
In Siciak \cite{Sic:1981}, it is shown that if
\begin{equation*}
    \log(|z| + 1) - c_u
    \leq
    u(z)
    \leq
    \log(|z| + 1) + c_u,
    \quad
    z \in \mathbb{C}^n,
\end{equation*}
for some constant $c_u$, then $R^a_{\mu, \delta} u(z) \leq \log(|z| + 1) + c_u$, for $\delta < e^{c_u}$, and
    consequently that $R^a_{\mu, \delta} u \in \mathcal{L}(\mathbb{C}^n)$.
It is also shown that $R^a_{\mu, \delta} u \searrow u$, as $\delta \searrow 0$.
The fact that $R^a_{\mu, \delta} u$ is plurisubharmonic is not obvious, but follows from geometric arguments.
We will recall some useful results on pseudoconvex domains before going through the proof.

Recall that $\Omega = \{z \in \mathbb{C}^n \,;\, u(z) < c\}$ is pseudoconvex if $u \in \mathcal{PSH}(\mathbb{C}^n)$.
To see this,
    assume $u$ is continuous and
    take compact $K \subset \Omega$.
Note that, $u \leq c - \delta$ on $K$,
    for some $\delta > 0$, so $u \leq c - \delta$ on $\widehat{K}_{\mathcal{PSH}(\mathbb{C}^n)}$, where
\begin{equation*}
  \widehat{K}_{\mathcal{PSH}(\mathbb{C}^n)}
  =
  \bigg \{
    z \in \mathbb{C}^n
      \,;\, 
    v(z) \leq \sup_{w \in K} u(w),
      \text{ for all }
      u \in \mathcal{PSH}(\mathbb{C}^n)
  \bigg \}
\end{equation*}
is the plurisubharmonically convex hull of $K$ in $\mathbb{C}^n$,
  which is compact.
So, we have that $\widehat{K}_{\mathcal{PSH}(\Omega)} \subset \widehat{K}_{\mathcal{PSH}(\mathbb{C}^n)} \subset \Omega$,
    which implies that $\widehat{K}_{\mathcal{PSH}(\Omega)}$ is relatively compact in $\Omega$.
Hence, $\Omega$ is pseudoconvex.
In the case that $u$ is not continuous,
    we consider the pseudoconvex sets
        $\Omega_\delta = \{z \in \mathbb{C}^n \,;\, u * \chi_\delta(z) < c\}$, for $\delta > 0$.
Since $\Omega_\delta \nearrow \Omega$, as $\delta \searrow 0$,
    we have, by the Behnke-Stein theorem,
    that $\Omega$ is pseudoconvex.

We call a continuous function $\mu \colon \mathbb{C}^n \rightarrow \mathbb{R}_+$ a \emph{distance function} if
    {\bf(i)}
        $\mu(z) = 0$ if and only if $z = 0$ and
    {\bf (ii)}
        $\mu(t z) = |t| \mu(z)$ for $z \in \mathbb{C}^n$ and $t \in \mathbb{C}$.
Note that, by {\bf(ii)}, we have constants $r_\mu = \inf_{|z| = 1} \mu(z)$ and $s_\mu = \sup_{|z| = 1} \mu(z)$
    such that $r_\mu |z| \leq \mu(z) \leq s_\mu |z|$, for $z \in \mathbb{C}^n$,
    and by {\bf(i)} we have that $r_\mu > 0$.
For an open $\Omega \subset \mathbb{C}^n$, we define the \emph{$\mu$-distance to the boundary} by
\begin{equation*}
    \mu_\Omega(z)
    =
    \inf_{w \not \in \Omega} \mu(z - w),
    \quad
    z \in \Omega,
\end{equation*}
and note that it is a continuous function.
Pseudoconvex domains are classically characterized by distances, namely
    an open $\Omega \subset \mathbb{C}^n$ is pseudoconvex if and only if
    $-\log \mu_\Omega$ is plurisubharmonic for every (or some) distance function $\mu$.
See Klimek~\cite[Thm.~2.10.4]{Kli:1991}.

\begin{proof}[Proof of Theorem \ref{thm:1.1}]
To start off we show that $R^a_{\mu, \delta} u$ is continuous and plurisubharmonic.
We let 
\begin{equation*}
    \Omega
    =
    \big \{
      (z, a) \in \mathbb{C}^n \times \mathbb{C} \,;\, |a| < e^{-u(z)}
    \big \}
\end{equation*}
and note that $\Omega$ is a sublevel set of $(z, a) \mapsto u(z) + \log |a|$, so it is pseudoconvex.
We define a distance function $\widehat{\mu}$ on $\mathbb{C}^{n + 1}$
    by $\widehat{\mu}(z, a) = |a| + \delta^{-1}\mu(z)$,
    and see that,
    for $(z, a) \in \mathbb{C}^n \times \mathbb{C}$,
\begin{equation*}
    \widehat{\mu}_\Omega(z, a)
    =
    \inf
    \big \{
        |a - b| + \delta^{-1}\mu(z - w) \,;\, (w, b) \in \mathbb{C}^n \times \mathbb{C}, |b| \geq e^{-u(w)}
    \big \}.
\end{equation*}
By \cite[Thm.~2.10.4]{Kli:1991}, we have that $-\log \widehat{\mu}_\Omega$ is continuous and plurisubharmonic on
    $\Omega \supset \mathbb{C}^n \times \{0\}$, so $\mathbb{C}^n \ni z \mapsto -\log \widehat{\mu}_\Omega(z, 0)$
    is continuous and plurisubharmonic and
\begin{equation*}
\begin{aligned}
    \widehat{\mu}_\Omega(z, 0)
    &=
    \inf
    \big \{
        |b| + \delta^{-1}\mu(z - w) \,;\, (w, b) \in \mathbb{C}^n \times \mathbb{C}, |b| \geq e^{-u(w)}
    \big \}
\\
    &=
    \inf_{w \in \mathbb{C}^n}
    \{
        e^{-u(w)} + \delta^{-1}\mu(z - w)
    \}
    =
    e^{-R^a_{\mu, \delta} u(z)},
    \qquad
    z \in \mathbb{C}^n.
\end{aligned}
\end{equation*}
So, $R^a_{\mu, \delta} u$ is plurisubharmonic and continuous.

Since $S$ is a lower set, we have,
    by \cite[Thm.~5.8]{MagSigSigSno:2023},
    that $H_S(z - w) \leq H_S(z) + \sigma_S |w|$.
Using that $e^{-H_S} \leq 1$, we get
\begin{equation*}
\begin{aligned}
    R^a_{\mu, \delta} u(z)
    &\leq
    -\log \inf_{w \in \mathbb{C}^n}
    \{
        e^{-H_S(z - w)} + \delta^{-1} e^{c_u}r_\mu |w|
    \}
    +
    c_u\\
    &\leq
    -\log \inf_{w \in \mathbb{C}^n}
    \{
        e^{-H_S(z) - \sigma_S |w|} + e^{-H_S(z)}\delta^{-1} e^{c_u}r_\mu |w|
    \}
    +
    c_u\\
    &\leq
    H_S(z)
    -
    \log \inf_{w \in \mathbb{C}^n}
    \{
        e^{-\sigma_S |w|} + \delta^{-1} e^{c_u}r_\mu |w|
    \}
    +
    c_u,
    \quad
    z \in \mathbb{C}^n.
\end{aligned}
\end{equation*}
Note that, if
    $f \colon \mathbb{R}_+ \rightarrow \mathbb{R}$,
    $f(x) = e^{-ax} + bx$ for $a, b > 0$,
    then $f'(x) = b - ae^{-ax} > b - a$.
So, $f$ is increasing if $a < b$.
We have that $\delta < \sigma_S^{-1}e^{c_u}r_\mu$, implying that $\sigma_S < \delta^{-1}e^{c_u}r_\mu$.
Hence, the last infimum is obtained at $w = 0$.
Therefore, $R^a_{\mu, \delta} u(z) \leq H_S(z) + c_u$ and $R^a_{\mu, \delta} u \in \mathcal{L}^S(\mathbb{C}^n)$.

We let $\delta_j \searrow 0$, and set $u_{\delta, w}(z) = -\log\big(e^{-u(z - w)} + \delta^{-1}\mu(w)\big)$ and
\begin{equation*}
    K_j = \{ w \in \mathbb{C}^n \,;\, u_{\delta_j, w}(z) \geq u(z) \}.
\end{equation*}
Since $u$ is bounded below, we have $\varliminf_{|z| \rightarrow \infty} u_{\delta_j, w}(z) = -\infty$,
    and therefore $K_j$ is bounded.
We also have that $u$ is upper semicontinuous, so $K_j$ is closed, and consequently compact.
We also have that $u_{\delta_j, w}(z)$ is upper semicontinuous as a function of $w$,
    so there exists $w_j \in K_j$ such that $R^a_{\mu, \delta_j} u(z) = u_{\delta_j, w_j}(z)$.
Note that, $u_{\delta, w}(z) \searrow -\infty$, as $\delta \searrow 0$,
    for all $z \in \mathbb{C}^n$ and $w \in \mathbb{C}^{n*}$.
So, $\cap_j K_j = \{0\}$, $w_j \rightarrow 0$ as $j \rightarrow \infty$, and
\begin{equation*}
    u(z)
    \leq
    \lim_{j \rightarrow \infty} R^a_{\mu, \delta_j} u(z)
    =
    \lim_{j \rightarrow \infty} u_{\delta_j, w_j}(z)
    \leq
    \varlimsup_{j \rightarrow \infty} u(z - w_j)
    \leq
    u(z),
    \quad
    z \in \mathbb{C}^n.
\end{equation*}
Hence, $R^a_{\mu, \delta} u \searrow u$, as $\delta \searrow 0$.
\end{proof}

The case where $\mu = | \cdot |$ appears in \cite[Prop.~3.1]{BayHusLevPer:2020}.
There it is claimed that the result holds when $S$ is a neighborhood of $0$ in $\mathbb{R}^n_+$.
This claim is false.
\begin{example} 
\label{ex:3.1}
Let
    $S = \operatorname{ch}\{(0, 0), (a, 0), (0, a), (b, a)\}$,
    $a > 0$,
    $b > a(a + 1)$, and
    $\mu = | \cdot |$.
Note that, $S = \operatorname{ch} (a\Sigma \cup \{(b, a)\})$,
    so $S$ contains a neighborhood of $0$ in $\mathbb{R}^2_+$.
We have
\begin{equation*}
    H_S(z_1, z_2)
    =
    \max
    \{
        b\log |z_1| + a\log |z_2|,
        a\log^+ \|z\|_\infty
    \},
    \quad
    z_1, z_2 \in \mathbb{C}^*.
\end{equation*}
Hence, with $r = b/(a + 1) > a$, we have that
\begin{equation*}
    H_S(\zeta, |\zeta|^{-r})
    \geq
    \log|\zeta|^{b - ra},
    \quad
    \zeta \in \mathbb{C}^*.
\end{equation*}
Setting $w = (\zeta, |\zeta|^{-r})$ in \eqref{eq:1.1}, with $\zeta \in \mathbb{C}^*$, yields
\begin{equation*}
\begin{aligned}
    R^a_{\mu, \delta} H_S(\zeta, 0)
    &\geq
    -\log(e^{-H_S(\zeta, |\zeta|^{-r})} + \delta^{-1}|\zeta|^{-r})
\\
    &\geq
    -\log(|\zeta|^{ar - b} + \delta^{-1}|\zeta|^{-r})
\\
    &=
    -\log(|\zeta|^{r(a + 1) - b} + \delta^{-1}) + r \log |\zeta|
\\
    &=
    -\log(1 + \delta^{-1}) + r \log |\zeta|.
\end{aligned}
\end{equation*}
So, $R^a_{\mu, \delta} H_S(\zeta, 0) - H_S(\zeta, 0) \geq (r - a) \log |\zeta| - \log(1 + \delta^{-1})$
    is not bounded above, since $r - a > 0$.
Hence, $R^a_{\mu, \delta} H_S$ is not in $\mathcal{L}^S(\mathbb{C}^n)$.
\end{example}

\section{Supremal convolutions}
\label{sec:4}
Let us recall some topological properties of $\mathcal{PSH}(\Omega)$, with the goal of determining when
    $\sup \mathcal{F} = \sup\{u \,;\, u \in \mathcal{F}\}$ is upper semicontinuous,
    for some locally upper bounded family $\mathcal{F} \subset \mathcal{PSH}(\Omega)$.
For more details, see Hörmander \cite[Thms.~4.1.8-9]{Hormander:LPDO}
    and \cite[Thm.~3.2.11-13]{Hormander:convexity}.
Recall that $\mathcal{PSH}(\Omega)$ is the family of plurisubharmonic functions on $\Omega$
    that are not identically $-\infty$ on any connected component of $\Omega$, so
\begin{equation*}
    \mathcal{PSH}(\Omega)
    \subset
    L^1_{\text{loc}}(\Omega)
    \subset
    \mathcal{D}'(\Omega).
\end{equation*}
The topology $\mathcal{PSH}(\Omega)$ inherits from the weak topology on $\mathcal{D}'(\Omega)$
    coincides with topology it inherits from
    $L_{\text{loc}}^1(\Omega)$,
    as a Fréchet space with semi-norms $f \mapsto \int_K |f|\,d\lambda$, for compact $K$.
With this topology, $\mathcal{PSH}(\Omega)$ is a closed subspace of $L_{\text{loc}}^1(\Omega)$,
    so it is a complete metrizable space.

Furthermore, $\mathcal{PSH}(\Omega)$ has a Montel property, which says that every sequence $u_1, u_2, \dots$
    in $\mathcal{PSH}(\Omega)$,
    that is locally bounded above and does not converge to $-\infty$ uniformly on every compact subset of $\Omega$,
    has a subsequence that converges
    in $\mathcal{PSH}(\Omega)$.

Sometimes, $\sup \mathcal{F} = \sup \mathcal{F}_0$ for some $\mathcal{F}_0 \subset \mathcal{F}$ containing a minimal
    element.
The Montel property then states that $\mathcal{F}_0$ is relatively compact.
If, in addition, $\mathcal{F}_0$ is closed then $\sup \mathcal{F}$ is upper semicontinuous.
This follows from Sigurdsson \cite[Prop.~2.1]{Sig:1991}, which we include for the convenience of the reader.
\begin{lemma}
\label{lem:3.1}
Let $\Omega \subset \mathbb{C}^n$ be open and $\mathcal{F}$ be a compact family in $\mathcal{PSH}(\Omega)$.
Then $\sup \mathcal{F}$ is upper semicontinuous and, consequently, plurisubharmonic.
\end{lemma}
\begin{proof}
To start off the proof, we recall some essential properties of integral averages.
For $f \in L_{\text{loc}}^1(\Omega)$,
    we let $M_r f(z)$ denote the integral average of $f$ over a euclidean ball with center $z$ and radius $r > 0$,
    where we assume that the euclidean distance from $z \in \Omega$ to the boundary of $\Omega$
        is strictly less than $r$.
We have that $L_{\text{loc}}^1(\Omega) \times \mathbb{C}^n \times \mathbb{R}_+^* \rightarrow \mathbb{R}$,
    $(f, z, r) \mapsto M_r f(z)$ is continuous. 

Let $u = \sup \mathcal{F}$, $z_0 \in \Omega$, and $a \in \mathbb{R}$ such that $u(z_0) < a$.
If we can show that there exists a neighborhood $V$ of $z_0$ such that $u|_V < a$ then $u$ is upper semicontinuous.
Let $\varepsilon > 0$ such that $u(z_0) < a - \varepsilon$ and $v_0 \in \mathcal{F}$,
    and take $r_0 > 0$ such that
\begin{equation*}
    v_0(z_0)
    \leq
    M_{r_0} v_0 (z_0)
    \leq
    a - \varepsilon.
\end{equation*}
By the continuity of $(f, z) \mapsto M_{r_0}f(z)$,
    there exists an open neighborhood $U_0$ of $v_0$ in $\mathcal{PSH}(\Omega)$
and an open neighborhood $V_0$ of $z_0$ such that
\begin{equation*}
    M_{r_0} v(z)
    <
    a - \varepsilon,
    \quad
    v \in U_0,\ z \in V_0.
\end{equation*}
The submean value property implies that $v(z) < a - \varepsilon$ for $v \in U_0$ and $z \in V_0$.
    Since $v_0$ was arbitrary and $\mathcal{F}$ is compact,
    there exists a finite covering $U_1, \dots, U_\ell$ of $\mathcal{F}$
    and open neighborhoods $V_1, \dots, V_\ell$ of $z_0$
    such that $v(z) < a - \varepsilon$, for $v \in U_j$ and $z \in V_j$.
If we set $V = \cap_j V_j$ then $v(z) < a - \varepsilon$, for all $v \in \mathcal{F}$ and $z \in V$.
So, $u|_V < a$ and $u$ is therefore upper semicontinuous.
\end{proof}

Let
    $0 \in S \subset \mathbb{R}^n_+$ be compact and convex,
    $1_n = (1, \dots, 1) \in \mathbb{N}^n$, and
    $\sigma_S = \varphi_S(1_n)$.
We will allow us a slight abuse of notation by identifying a vector in $\mathbb{C}^n$,
    denoted by a lower case letter,
    with a diagonal matrix,
    denoted by the corresponding upper case letter.
Thus, we identify $a \in \mathbb{C}^n$ with the diagonal matrix $A$ that has $a$ as its diagonal.
In this notation, the subadditivity of the supporting function of $S$ implies that
\begin{equation}
\label{eq:3.1}
\begin{aligned}
    H_S(Zw)
    &=
    H_S(Wz)
    =
    H_S(z_1 w_1, \dots, z_n w_n)
\\
    &\leq
    H_S(z) + H_S(w),
    \qquad
    z, w \in \mathbb{C}^n.
\end{aligned}
\end{equation}

\begin{proof}[Proof of Theorem \ref{thm:1.3}]
Recall that we are studying the operator defined by
\begin{equation*}
\begin{aligned}
    R^b_\delta u(z)
    &=
    \sup_{w \in \mathbb{C}^n} 
    \big \{
        u(Zw) - \delta^{-1} \log (\|w - 1_n\|_\infty + 1)
    \big \}
\\
    &=
    \log
    \sup_{w \in \mathbb{C}^n} 
        \frac{e^{u(Zw)}}{(\|w - 1_n\|_\infty + 1)^{1/\delta}},
    \qquad
    z \in \mathbb{C}^n.
\end{aligned}
\end{equation*}
With $u_{\delta, w}(z) = u(Zw) - \delta^{-1} \log (\|w - 1_n\|_\infty + 1)$, which is plurisubharmonic,
    we have that $R^b_\delta u = \sup_{w \in \mathbb{C}^n} u_{\delta, w}$.
The assumption that $u$ is locally bounded ensures that $u_{\delta, w} \in \mathcal{PSH}(\mathbb{C}^n)$.
For every $\delta > 0$, we have that $u_{\delta, 1_n} = u$, so $R^b_\delta u \geq u$.
Additionally,
\begin{equation*}
    H_S(w) \leq \sigma_S \log^+ \|w\|_\infty \leq \sigma_S \log (\|w - 1_n\|_\infty + 1),
    \quad
    w \in \mathbb{C}^n.
\end{equation*}
This, along with (\ref{eq:3.1}), implies that if $u \leq H_S + c_u$ then
\begin{equation*}
    R^b_\delta u
    \leq
    H_S + \sup_{w \in \mathbb{C}^n} \big\{(\sigma_S - \delta^{-1}) \log (\|w - 1_n\|_\infty + 1)\big\} + c_u.
\end{equation*}
Since $\delta < \sigma_S^{-1}$, we have $R^b_\delta u \leq H_S + c_u$.
We also have, for all $\gamma > 0$ and $z \in \mathbb{C}^n$,
\begin{equation*}
    R^b_\delta u(z)
    \leq
    \max
    \bigg \{
        \sup_{w \in 1_n + \gamma \overline{\mathbb{D}}^n} u(Zw),
        H_S(z) + c_u + (\sigma_S - \delta^{-1}) \log(\gamma + 1)
    \bigg \}.
\end{equation*}
With $z \in \mathbb{C}^n$ fixed and $\varepsilon > 0$, we may take $\gamma$ small enough that
    $u(Zw) \leq u(z) + \varepsilon$ for all $w \in 1_n + \gamma \overline{\mathbb{D}}^n$,
    since $u$ is upper semicontinuous.
Consequently, 
\begin{equation*}
    R^b_\delta u
    \leq
    \max
    \big \{
        u + \varepsilon,
        H_S + c_u + (\sigma_S - \delta^{-1}) \log(\gamma + 1)
    \big \}.
\end{equation*}
The second term converges to $-\infty$, as $\delta \searrow 0$.
Hence, since $R^b_\delta u \geq u$,
    we have that $R^b_\delta u \searrow u$, as $\delta \searrow 0$.

We now observe that $R^b_\delta u = \sup \mathcal{F}_{u, \delta}$, for
    $\mathcal{F}_{u, \delta} = \{\max \{u, u_{\delta, w}\} \,;\, w \in \mathbb{C}^n\}$, and
    $\mathcal{F}_{u, \delta} \subset \mathcal{PSH}(\mathbb{C}^n)$.
All functions in $\mathcal{F}_{u, \delta}$ are bounded above by $H_S + c_u$ and $u$ is a minimal element of
    $\mathcal{F}_{u, \delta}$,
    so $\mathcal{F}_{u, \delta}$ is relatively compact by the Montel property.
To show that $\mathcal{F}_{u, \delta}$ is closed in $\mathcal{PSH}(\mathbb{C}^n)$,
    and thus that $R^b_\delta u$ is plurisubharmonic, we note that
\begin{equation*}
    u_{\delta, w}
    \leq H_S + c_u + (\sigma_S - \delta^{-1})\log^+\|w\|_\infty,
    \quad
    w \in \mathbb{C}^n,
\end{equation*}
and recall as well that $u$ is assumed to be locally bounded below.
So, fixing $a \in \mathbb{C}^n$ and $U = B(z, 1)$ as the euclidean ball with center $z$ and radius $1$,
    and setting $M_2 < M_1$ as constants such that $M_2 \leq u(z) - c_u$ and $H_S(z) \leq M_1$, for $z \in U$,
    where $c_u > 0$ is such that $u \leq H_S + c_u$, we have
\begin{equation*}
    u_{\delta, w}(z) - u(z)
    \leq M_1 - M_2 + (\sigma_S - \delta^{-1})\log^+\|w\|_\infty,
    \quad
    z \in U,\ w \in \mathbb{C}^n.
\end{equation*}
So, if $w \in \mathbb{C}^n$ is such that $\log^+\|w\|_\infty \geq (M_1 - M_2)/(\delta^{-1} - \sigma_S) > 0$,
    then $u_{\delta, w}(z) \leq u(z)$, for $z \in U$.
Hence, since $u \leq R^b_\delta u$, we have
\begin{equation*}
    R^b_\delta u(z)
    =
    \sup_{w \in \bar{B}(z, r)} u_{\delta, w}(z),
    \quad
    z \in U,
\end{equation*}
where $r = \max\{1, e^{(M_1 - M_2)/(\delta^{-1} - \sigma_S)}\}$.
If we take a sequence $(w_j)_{j \in \mathbb{N}}$ in $\mathbb{C}^n$ such that $v_{\delta, w_j} \rightarrow v$ in
    $\mathcal{PSH}(\mathbb{C}^n)$,
    where $v_{\delta, w} = \max\{u, u_{\delta, w}\}$,
    we need to show that $v \in \mathcal{F}_{u, \delta}$.
If $(w_j)_{j \in \mathbb{N}}$ is unbounded then we can pick a subsequence
    $(w_{j_k})_{k \in \mathbb{N}} \subset \mathbb{C}^n \setminus \overline{B}(z, r)$.
Since $v_{\delta, w_{j_k}} = u$, we have that $v = u \in \mathcal{F}_{u, \delta}$.
If $(w_j)_{j \in \mathbb{N}}$ is bounded we can pick a convergent subsequence $(w_{j_k})_{k \in \mathbb{N}}$
    with limit $w$.
By the upper semicontinuity of $(z, w) \mapsto v_{\delta, w}$, we have that
    $\varlimsup_{k \rightarrow \infty} v_{\delta, w_{j_k}} = v_{\delta, w}$,
    and, by Hörmander \cite[Thm.~4.1.9]{Hormander:LPDO}, we have that
    $\varlimsup_{k \rightarrow \infty} v_{\delta, w_{j_k}} = v$ almost everywhere.
So, $v = v_{\delta, w}$ almost everywhere and, since they are plurisubharmonic,
    we have $v = v_{\delta, w} \in \mathcal{F}_{u, \delta}$.
Consequently, $\mathcal{F}_{u, \delta}$ is compact and, by Lemma \ref{lem:3.1}, $R^b_\delta u$ is plurisubharmonic.

All that remains to be proven is that $R^b_\delta u$ is lower semicontinuous.
By a change of variables, we have
\begin{equation*}
\begin{aligned}
    R^b_\delta u(z)
    &=
    \sup_{w \in \mathbb{C}^n}
    \{
        u(Zw) - \delta^{-1}\log(\|w - 1_n\|_\infty + 1)
    \}\\
    &=
    \sup_{w \in \mathbb{C}^n}
    \{
        u(w) - \delta^{-1}\log (\|Z^{-1}w - 1_n\|_\infty + 1)
    \},
    \quad
    z \in \mathbb{C}^n.
\end{aligned}
\end{equation*}
Setting $v_{\delta, w}(z) = u(w) - \delta^{-1} \log (\|Z^{-1}w - 1_n\|_\infty + 1)$,
    we note that $v_{\delta, w}$ is continuous on $\mathbb{C}^{*n}$ and
    $R^b_\delta u(z) = \sup_{w \in \mathbb{C}^n} v_{\delta, w}(z)$, for $z \in \mathbb{C}^{*n}$.
So, $R^b_\delta u$ is given as the supremum of continuous functions on the open set $\mathbb{C}^{*n}$,
    and is consequently lower semicontinuous on $\mathbb{C}^{*n}$.
\end{proof}

\section{Integral convolutions over diagonal matrices}
\label{sec:5}
Recall that the regularization from Theorem \ref{thm:1.4} is given by
\begin{equation*}
\begin{aligned}
    R^c_\delta u(z)
    &=
    \int_{\mathbb{C}^n} u(Az) \psi_\delta(A)\,d\lambda(A)
    =
    \int_{\mathbb{C}^n} u\big((I + \delta B)z\big) \psi(B)\,d\lambda(B)
\\
    &=
    \int_{\mathbb{C}^n} u\big((1 + \delta w_1) z_1, \dots, (1 + \delta w_n) z_n\big) \psi(w)\,d\lambda(w),
    \quad
    z \in \mathbb{C}^n, \nonumber
\end{aligned}
\end{equation*}
where $\psi \in \mathcal{C}_0^\infty(\mathbb{C}^n)$ is rotationally symmetric in each variable and
    $\int_{\mathbb{C}^n} \psi\,d\lambda = 1$,
and $\psi_\delta(z) = \delta^{-2n} \psi\big((z - 1_n)/\delta\big)$.
\begin{proof}[Proof of Theorem \ref{thm:1.4}]
By the Fubini-Tonelli theorem,
    $R^c_\delta \colon L_{\text{loc}}^1(\mathbb{C}^n) \rightarrow L_{\text{loc}}^1(\mathbb{C}^n)$
    and $R^c_\delta u \rightarrow u$, in the $L_{\text{loc}}^1$ topology, as $\delta \searrow 0$.
Also, $R^c_\delta \colon \mathcal{PSH}(\mathbb{C}^n) \rightarrow \mathcal{PSH}(\mathbb{C}^n)$.
By \cite[Prop.~3.2]{MagSigSigSno:2023}, we have $H_S(1_n + \delta w) \leq \delta \sigma_S$,
    for all $w \in \overline{\mathbb{D}}^n$,
    so $R^c_\delta \colon \mathcal{L}^S(\mathbb{C}^n) \rightarrow \mathcal{L}^S(\mathbb{C}^n)$,
    for all compact convex $S \subset \mathbb{R}_+^n$ containing $0$.
In fact, if $u \in \mathcal{L}^S(\mathbb{C}^n)$ and $c_u$ is a constant such that $u \leq H_S + c_u$ then,
    by (\ref{eq:3.1}),
\begin{equation*}
\begin{aligned}
    R^c_\delta u(z)
    &\leq
    \int_{\mathbb{C}^n} (H_S(z) + H_S(1_n + \delta w) + c_u)\psi(w)\,d\lambda(w)
\\
    &=
    H_S(z) + C\sigma_S \delta + c_u,
    \qquad
    z \in \mathbb{C}^n,
\end{aligned}
\end{equation*}
for $C = \sup\limits_{w \in \operatorname{supp} \psi} \|w\|_\infty$.
Since the map $A \mapsto Az$ has the Jacobi determinant $|z_1 \cdots z_n|^2$,
\begin{equation*}
    R^c_\delta u(z)
    =
    \int_{\mathbb{C}^n} u(w) \psi_\delta(Z^{-1}w)|z_1 \cdots z_n|^{-2}\,d\lambda(w),
    \quad
    z \in \mathbb{C}^{*n}.
\end{equation*}
By applying Lebesgue's theorem on dominated convergence, we may differentiate by $z_j$ under the
    integral sign arbitrarily often.
So, if $u \in L_{\text{loc}}^1(\mathbb{C}^n)$ then $R^c_\delta u \in \mathcal{C}^\infty(\mathbb{C}^{*n})$.

Note that, $R^c_\delta u(z) = U_z(\psi_\delta)$, where $U_z$ is the distribution associated to $u_z$,
    which is locally integrable since $u$ is locally bounded.
Let us fix $a \in \mathbb{C}^n$.
By assumption, $u_z \rightarrow u_a$ in $\mathcal{PSH}(\mathbb{C}^n)$, as $z \rightarrow a$, so
\begin{equation*}
    \lim_{z \rightarrow a}
        R^c_\delta u(z)
    =
    \lim_{z \rightarrow a}
        U_z(\psi_\delta)
    =
    U_a(\psi_\delta)
    =
    R^c_\delta u(a).
    \qedhere
\end{equation*}
\end{proof}

Recall that the operator in Theorem \ref{thm:1.5} is given by
\begin{equation*}
    R^d_\delta u(z)
    =
    \log R^c_\delta e^u (z)
    =
    \log \int_{\mathbb{C}^n} e^{u(Za)} \psi_\delta(a)\,d\lambda(a),
    \quad
    z \in \mathbb{C}^n.
\end{equation*}
We have already shown that $e^{R^d_\delta u} \searrow e^u$, as $\delta \searrow 0$.
We also have
\begin{equation*}
    R^d_\delta u
    \leq
    H_S + c_u + \log \int_{\mathbb{C}^n} e^{H_S(a)} \psi_\delta(A)\,d\lambda(A),
\end{equation*}
if $c_u$ is a constant such that $u \leq H_S + c_u$.
To show that $R^d_\delta u$ is plurisubharmonic, we recall that following result.
\begin{lemma}
\label{lemma:4.2}
Let $\Omega \subset \mathbb{C}$ and $u \colon \Omega \rightarrow [0, \infty[$.
Then $\log u \in \mathcal{SH}(\Omega)$ if and only if
    $z \mapsto u(z) e^{2 \operatorname{Re}(\tau z)} \in \mathcal{SH}(\Omega)$,
    for all $\tau \in \mathbb{C}$.
\end{lemma}
\begin{proof}
To simplify notation, we denote by $\operatorname{Re}(\tau \cdot)$ the function $z \mapsto \operatorname{Re}(\tau z)$.
If $\log u$ is subharmonic then $e^{\log u + 2 \operatorname{Re}(\tau \cdot)}$ is as well,
    since $2 \operatorname{Re}(\tau \cdot)$ is subharmonic, for all $\tau \in \mathbb{C}$.

Now assume $u e^{2 \operatorname{Re}(\tau \cdot)}$ is subharmonic, for all $\tau \in \mathbb{C}$.
We further more assume $u$ is smooth and $u > 0$.
We have
\begin{equation*}
    \Delta \log u
    =
    4 \frac{\partial}{\partial z}
    \left (
        \frac{1}{u} \frac{\partial u}{\partial \bar z}
    \right )
    =
    \frac{\Delta u}{u} - \frac{4}{u^2} \frac{\partial u}{\partial z} \frac{\partial u}{\partial \bar z}
    =
    \frac{u \Delta u - |\nabla u|^2}{u^2}.
\end{equation*}
For all $\tau \in \mathbb{C}$, we also have
\begin{equation*}
\begin{aligned}
    0
    \leq
    \Delta (u e^{2 \operatorname{Re}(\tau \cdot)})
    &=
    4 \frac{\partial}{\partial z} \frac{\partial}{\partial \bar z} (u e^{2 \operatorname{Re}(\tau \cdot)})
    =
    4 e^{2 \operatorname{Re}(\tau \cdot)} \frac{\partial}{\partial z}
    \left ( 
        \frac{\partial u}{\partial \bar z}
        +
        \bar \tau u
    \right )
\\
    &=
    4u e^{2 \operatorname{Re}(\tau \cdot)}
    \left (
        \frac{\Delta u}{4u}
        +
        \frac{\tau}{u}
        \frac{\partial u}{\partial \bar z}
        +
        \frac{\bar \tau}{u}
        \frac{\partial u}{\partial z}
        +
        |\tau|^2
    \right )
\\
    &=
    4u e^{2 \operatorname{Re}(\tau \cdot)}
    \left (
        \frac{\Delta u}{4u}
        +
        \left |
            \tau + \frac{1}{u} \frac{\partial u}{\partial z}
        \right |^2
        -
        \frac{1}{u^2} \frac{\partial u}{\partial z} \frac{\partial u}{\partial \bar z}
    \right ).
\end{aligned}
\end{equation*}
As this holds for all $\tau \in \mathbb{C}$ and $u > 0$, we have
\begin{equation*}
    0
    \leq
    \frac{\Delta u}{4u}
    -
    \frac{1}{u^2} \frac{\partial u}{\partial z} \frac{\partial u}{\partial \bar z}
    =
    \Delta \log u,
\end{equation*}
so $\log u$ is subharmonic.
For a general $u \geq 0$, we set $v_\delta = (u + \delta) * \chi_\delta > 0$,
    where $\chi_\delta$ is a standard smoothing kernel.
Then, for $z \in \mathbb{C}$,
\begin{equation*}
    \int_{B(z, r)}
        v_\delta(w) e^{2 \operatorname{Re}(\tau w)}\,d\lambda(w)
    \geq
    \int_{B(z, r)}
        (u * \chi_\delta)(w) e^{2 \operatorname{Re}(\tau w)}\,d\lambda(w) + \delta e^{2 \operatorname{Re}(\tau z)}
\end{equation*}
and, by the Fubini-Tonelli theorem,
\begin{equation*}
    \int_{B(z, r)}
        (u * \chi_\delta)(w) e^{2 \operatorname{Re}(\tau w)}\,d\lambda(w)
    =
    \int_{\mathbb{C}}
        \chi_\delta(a)
    \int_{B(z, r)}
        e^{2 \operatorname{Re}(\tau w)}
        u(w - a)\,d\lambda(w)d\lambda(a)
\end{equation*}
\begin{equation*}
    \geq
    e^{2 \operatorname{Re}(\tau z)}
    \int_{\mathbb{C}}
        u(z - a)
        \chi_\delta(a)
        \,d\lambda(a)
    =
    u * \chi_\delta(z)
    e^{2 \operatorname{Re}(\tau z)}.
\end{equation*}
So, $v_\delta e^{2 \operatorname{Re}(\tau \cdot)}$ satisfies the submean inequality, and is therefore subharmonic.
We have already shown that this implies that $\log v_\delta$ is subharmonic,
    and since $\log v_\delta \searrow \log u$, as $\delta \searrow 0$, $\log u$ must also be subharmonic.
See Ransford \cite[Thm.~2.4.6]{Ransford:1995}.
\end{proof}

\begin{proof}[Proof of Theorem \ref{thm:1.5}]
What remains is to show that $R^d_\delta u$ is plurisubharmonic.
We fix $a, b \in \mathbb{C}^n$, $\tau \in \mathbb{C}$, and set, for $\zeta \in \mathbb{C}$,
\begin{equation*}
    v(\zeta)
    =
    R^c_\delta e^u(a + \zeta b) e^{2 \operatorname{Re}(\tau \zeta)}
    =
    \int_{\mathbb{C}^n}
        e^{u(W(a + \zeta b)) + 2 \operatorname{Re}(\tau \zeta)}\psi_\delta(W)\,d\lambda(W).
\end{equation*}
By \cite[Thm.~2.4.8]{Ransford:1995}, we have that $v \in \mathcal{SH}(\mathbb{C})$.
So, by Lemma \ref{lemma:4.2}, $\zeta \mapsto R^d_\delta u(a + \zeta b)$ is subharmonic and, consequently,
    $R^d_\delta u$ is plurisubharmonic.
\end{proof}

\section{Uniform continuity in Lelong classes}
\label{sec:6}
Classically, sufficient conditions on $K$ such that $V_K$ is Hölder continuous on $\mathbb{C}^n$ have been studied.
See, for example, Siciak \cite{Sic:1997b}.
Perera \cite{Per:2024}, continued this study for $V^S_{K, q}$.
Her main result, stated in the abstract, claims that if
    $S$ is a convex body, and
    $K$ and $q$ are sufficiently regular
    then $V^S_{K, q}$ is $\alpha$-Hölder continuous on $\mathbb{C}^n$.
The main aim of this section is to show that this result can not be true if $S$ is not a lower set,
    no matter what regularity we impose on $K$ and $q$.
We will do this by showing that if $S$ is not a lower set then $V^S_{K, q}$ is not uniformly continuous.
Let us begin by recalling the definition of some classes of regularity.

Let $U \subset \mathbb{C}^n$ be an open set and $f$ be a function on $U$.
We say that $f$ is \emph{uniformly continuous} if for $\varepsilon > 0$ there exists $\delta > 0$ such that
    $|f(x + y) - f(x)| < \varepsilon$, for all $x, y \in U$, such that $x + y \in U$ and $|y| < \delta$.
We say $f$ is \emph{$\alpha$-Hölder continuous}, for $0 < \alpha \leq 1$, if
\begin{equation*}
    |f(x) - f(y)|
    \leq
    C|x - y|^\alpha,
    \quad
    x, y \in U,
\end{equation*}
for some constant $C$.
We say that $f$ is \emph{Hölder continuous} if it is $\alpha$-Hölder continuous for some $0 < \alpha \leq 1$.
If $f$ is $\alpha$-Hölder continuous for $\alpha = 1$ we say that it is \emph{Lipschitz continuous} with
    \emph{Lipschitz constant} $C$.
Note that, Hölder continuous functions are uniformly continuous.
We can describe $H_S$ in these terms, depending on whether $S$ is a lower set or not.
\begin{theorem}
\label{thm:5.1}
Let $0 \in S \subset \mathbb{R}^n_+$ be compact and convex.
If $S$ is a lower set then $H_S$ is Lipschitz continuous with Lipschitz constant $\sigma_S = \varphi_S(1_n)$.
If $S$ is not a lower set then, for $\delta > 0$, there exists $w \in \delta \overline{\mathbb{D}}^n$
    such that the function
\begin{equation}
\label{eq:5.1}
    z \mapsto H_S(z + w) - H_S(z)
\end{equation}
    is not bounded above.
Hence, $H_S$ is not uniformly continuous.
\end{theorem}
\begin{proof}
First, we assume $S$ is a lower set.
Then, by \cite[Thm.~5.8]{MagSigSigSno:2023},
\begin{equation*}
    H_S(z)
    =
    \varphi_S(\log^+|z_1|, \dots, \log^+|z_n|),
    \quad
    z \in \mathbb{C}^n,
\end{equation*}
Note that, $\log^+ x$ is Lipschitz continuous with Lipschitz constant $1$ and
    $\varphi_S$ is Lipschitz continuous with Lipschitz constant $\sigma_S$.
To see this, note that
\begin{equation*}
    \varphi_S(\xi)
    =
    \varphi_S(\xi - \eta + \eta)
    \leq
    \varphi_S(\xi - \eta) + \varphi_S(\eta),
    \quad
    \xi, \eta \in \mathbb{R}^n,
\end{equation*}
so
\begin{equation*}
    \varphi_S(\xi) - \varphi_S(\eta)
    \leq
    \varphi_S(\xi - \eta)
    \leq
    \sigma_S \|\xi - \eta\|_\infty
    \leq
    \sigma_S |\xi - \eta|,
    \quad
    \xi, \eta \in \mathbb{R}^n.
\end{equation*}
With $\operatorname{Log}^+ z = (\log^+|z_1|, \dots, \log^+|z_n|)$, we have
\begin{equation*}
\begin{aligned}
    |H_S(z) - H_S(w)|
    &=
    |\varphi_S(\operatorname{Log}^+ z) - \varphi_S(\operatorname{Log}^+ w)|
\\
    &\leq
    \sigma_S |\operatorname{Log}^+ z - \operatorname{Log}^+ w|
    \leq
    \sigma_S |z - w|,
    \quad
    z, w \in \mathbb{C}^n.
\end{aligned}
\end{equation*}

Before continuing, recall that $\varphi_S \leq \varphi_T$ if and only if $S \subset T$,
    for convex $S, T \subset \mathbb{R}^n$.
So, $H_S \leq H_T$ if and only if $S \subset T$.
It will prove useful to us that if $S$ is not a subset of $T$ then $H_S - H_T$ is not bounded above.

Now assume $S$ is not a lower set.
Then, after possibly rearranging the variables,
    there exists $s = (s_1, \dots, s_n) = (s', s'') \in \mathbb{R}^\ell_+ \times \mathbb{R}^{n - \ell}_+$ such that
    $s \in S$ and $(s', 0) \not \in S$.
By \cite[Prop.~3.3]{MagSigSigSno:2023}, we have that $H_S(z', 0) = H_T(z')$, for $z' \in \mathbb{C}^\ell$,
    where $0 \in T \subset \mathbb{R}^\ell_+$ is compact convex and $s' \not \in T$.
So, with $L = \operatorname{ch}\{0, s'\} \subset \mathbb{R}^\ell_+$, we have that $H_L - H_T$ is not bounded above.
Note that, for $z' \in \mathbb{C}^{*\ell} \setminus \mathbb{D}^\ell$,
\begin{equation*}
\begin{aligned}
    H_S(z', \delta, \dots, \delta)
    &= \sup_{x \in S}
        \big (x_1 \log|z'_1| + \dots + x_\ell\log|z'_\ell| + (x_{\ell + 1} + \dots + x_n) \log \delta \big )
\\
    &\geq s_1 \log|z'_1| + \dots + s_\ell\log|z'_\ell| + (s_{\ell + 1} + \dots + s_n) \log \delta
\\
    &= H_L(z') + (s_{\ell + 1} + \dots + s_n) \log \delta.
\end{aligned}
\end{equation*}
\makebox[0pt][l]{So,}
\begin{equation*}
    H_S(z', \delta, \dots, \delta) - H_S(z', 0, \dots, 0)
    \geq
    H_L(z') - H_T(z') + C_\delta,
    \quad
    z' \in \mathbb{C}^\ell \setminus \mathbb{D}^\ell,
\end{equation*}
where $C_\delta = (s_{\ell + 1} + \dots + s_n) \log \delta$.
Consequently, \eqref{eq:5.1} is not bounded above for
    $w = (0, \dots, 0, \delta, \dots, \delta) \in \delta \overline{\mathbb{D}}^n$.
To be clear, the first $\ell$ elements of $w$ are $0$.
\end{proof}
We have,
    by \cite[Prop.~4.3]{MagSigSigSno:2023},
    that $V^S_{\mathbb{T}^n} = V^S_{\overline{\mathbb{D}}^n} = H_S$,
    where $\mathbb{T} = \{z \in \mathbb{C} \,;\, |z| = 1\}$.
However, the previous proposition does not provide a counterexample to the main result of Perera \cite{Per:2024},
    since $\mathbb{D}^n$ does not satisfy the boundary condition required by the Proposition and $\mathbb{T}^n$
    is not the closure of an open set.
A consequence of the previous result will, however, suffice.
\begin{corollary}
\label{cor:5.2}
Let $0 \in S \subset \mathbb{R}^n_+$ be compact and convex.
If $S$ is not a lower set then $\mathcal{L}^S_+(\mathbb{C}^n)$ contains no uniformly continuous functions.
\end{corollary}
\begin{proof}
Let $\delta > 0$, $u \in \mathcal{L}^S_+(\mathbb{C}^n)$, and $c_u$ be a constant such that
\begin{equation*}
    H_S(z) - c_u
    \leq
    u(z)
    \leq
    H_S(z) + c_u,
    \quad
    z \in \mathbb{C}^n.
\end{equation*}
By Theorem \ref{thm:5.1}, there exists $w \in \delta \overline{\mathbb{D}}^n$ such that
    the function given by \eqref{eq:5.1} is not bounded above.
Since
\begin{equation*}
    u(z + w) - u(z)
    \geq
    H_S(z + w) - H_S(z) - 2c_u,
    \quad
    z, w \in \mathbb{C}^n,
\end{equation*}
the function $z \mapsto u(z + w) - u(z)$ is also not bounded above.
Hence, $u$ is not uniformly continuous.
\end{proof}
By \cite[Prop.~4.5]{MagSigSigSno:2023}, we have that $V^{S*}_{K, q} \in \mathcal{L}^S_+(\mathbb{C}^n)$,
    for all compact convex $S \subset \mathbb{R}^n_+$ containing $0$,
    compact $K \subset \mathbb{C}^n$, and
    admissible $q$ on $K$.
So, if $S$ is not a lower set then $V^S_{K, q}$ can not be uniformly continuous,
    and thus it can not be Hölder continuous.

There is an error in the proof of Lemma 2.4 in \cite{Per:2024}.
A major step in the proof involves finding a constant $C_w$, depending on $w$,
    such that $V^S_{\mathbb{C}^n, q}(z + w) \leq H_S(z) + C_w$,
    for $z \in \mathbb{C}^n$,
    where $q$ has been extended to $\mathbb{C}^n$ by standard methods of extending Hölder continuous functions.
This is in contradiction to \cite[Thm.~5.8]{MagSigSigSno:2023},
    where it is shown that $\mathcal{L}^S(\mathbb{C}^n)$ is translation invariant if and only if $S$ is a lower set.

The precise location of the error is the third step of the large inequality of Case 1,
    where it is incorrectly stated that
\begin{equation*}
    H_S(z_1, z_2)
    =
    \sup_{x \in S}
        x_2 \log |z_2|,
    \quad
    z_1 \in \mathbb{D},\ z_2 \in \mathbb{C} \setminus \mathbb{D}.
\end{equation*}
By \cite[Thm.~5.8]{MagSigSigSno:2023}, this can only hold when $S$ is a lower set.
We can also show that it fails by a direct counterexample.

\begin{example}
Let $S = \operatorname{ch}\{(0, 0), (1, 1), (1, 0)\}$.
Then $\sup_{x \in S} x_2 \log|z_2| = \log|z_2|$.
However, if $|z_2| > 1$, then
$
    H_S(z_2^{-1}, z_2)
    =
    \sup_{x \in S} (x_2 - x_1) \log |z_2| = 0,
$
since $x_2 - x_1 \leq 0$, for all $x \in S$.
\end{example}

It is worth noting that other results in \cite{Per:2024} need correcting,
    as they depend on wrong results from other papers.
Corollary~$1.3$ depends on \cite[Prop.~3.1]{BayHusLevPer:2020}, which is shown to be incorrect in
    Example~\ref{ex:3.1}, herein.
Proposition $2.3$ also depends on \cite[Prop.~3.1]{BayHusLevPer:2020},
    as well as depending on Levenberg and Perera \cite[Prop.~2.2]{LevPer:2020}.
The proof of Proposition 2.2 in \cite{LevPer:2020} involves constructing a strictly plurisubharmonic function in
    $\mathcal{L}^S_+(\mathbb{C}^n)$.
First, it is done under the assumption that $S$ is a convex polytope,
    that is $S = \operatorname{ch}\{v_1, \dots, v_\ell\} \subset \mathbb{R}^n_+$,
    where $v_1 = 0$.
The specific function chosen is
\begin{equation*}
    h_S(z)
    =
    \log
        \Big (
        1
        +
        \sum_{j = 2}^\ell |z|^{v_j}
        \Big ),
    \quad
    z \in \mathbb{C}^n.
\end{equation*}
In the case where $S$ is not a convex polytope but still contains a neighborhood of $0$,
    a decreasing sequence of convex polytopes $S_1, S_2, \dots$ is taken such that
    $\cap_{j = 1}^\infty S_j = S$.
They then claim that $h_{S_j}$ is a decreasing sequence and use its limit as a candidate function.
But the sequence is not decreasing.
This can be seen by considering the point $1_n = (1, \dots, 1)$.
In fact, $h_{S_j}(1_n) = \log(\# \operatorname{ext} S_j)$,
    where $\# A$ denotes the number of elements in the set $A$ and
    $\operatorname{ext} B$ denotes the set of extremal points of the convex set $B$.
So,
    if $S$ is not a convex polytope,
    then $h_{S_j}(1_n) \rightarrow +\infty$, as $j \rightarrow +\infty$.
An alternative to \cite[Prop.~2.2]{LevPer:2020} can be found in
    \cite[Lemma 2.1 and Prop.~2.2]{Sno:2024a},
    which apply when $S$ contains a neighborhood of $0$.

\bibliographystyle{siam}
\bibliography{rs_bibref}

\begin{thebibliography}{10}

\bibitem{BayHusLevPer:2020}
{\sc T.~Bayraktar, S.~Hussung, N.~Levenberg, and M.~Perera}, {\em
  Pluripotential theory and convex bodies: a {S}iciak-{Z}aharjuta theorem},
  Comput. Methods Funct. Theory, 20 (2020), pp.~571--590.

\bibitem{Fer:1973}
{\sc J.~P. Ferrier}, {\em Spectral theory and complex analysis}, vol.~No. 4 of
  North-Holland Mathematics Studies, North-Holland Publishing Co.,
  Amsterdam-London; American Elsevier Publishing Co., Inc., New York, 1973.
\newblock Notas de Matem\'{a}tica, No. 49. [Mathematical Notes].

\bibitem{Hal:1994}
{\sc S.~Halvarsson}, {\em Extension of entire functions with controlled
  growth}, Math. Scand., 74 (1994), pp.~73--97.

\bibitem{Hal:1996a}
\leavevmode\vrule height 2pt depth -1.6pt width 23pt, {\em A geometric
  interpretation of the relative order between entire functions}, Arch. Math.
  (Basel), 66 (1996), pp.~197--206.

\bibitem{Hal:1996b}
\leavevmode\vrule height 2pt depth -1.6pt width 23pt, {\em Growth properties of
  entire functions depending on a parameter}, Ann. Polon. Math., 64 (1996),
  pp.~71--96.

\bibitem{Hormander:LPDO}
{\sc L.~H{\"o}rmander}, {\em {The analysis of linear partial differential
  operators I and II}}, Springer Verlag, New York Berlin Heidelberg, 1983.

\bibitem{Hormander:convexity}
\leavevmode\vrule height 2pt depth -1.6pt width 23pt, {\em {Notions of
  convexity}}, vol.~127 of {Progress in Mathematics}, Birkh{\"a}user Boston
  Inc., Boston, MA, 1994.

\bibitem{Kis:1993}
{\sc C.~O. Kiselman}, {\em Order and type as measures of growth for convex or
  entire functions}, Proc. London Math. Soc. (3), 66 (1993), pp.~152--186.

\bibitem{Kli:1991}
{\sc M.~Klimek}, {\em Pluripotential theory}, vol.~6 of London Mathematical
  Society Monographs. New Series, The Clarendon Press Oxford University Press,
  New York, 1991.

\bibitem{LevPer:2020}
{\sc N.~Levenberg and M.~Perera}, {\em A global domination principle for
  {$P-$}pluripotential theory}, in Complex analysis and spectral theory,
  vol.~743 of Contemp. Math., Amer. Math. Soc., Providence, RI, 2020,
  pp.~11--19.

\bibitem{MagSigSig:2023}
{\sc B.~S. Magn{\'u}sson, {\'A}.~E. Sigurðard{\'o}ttir, and R.~Sigurðsson},
  {\em Polynomials with exponents in compact convex sets and associated
  weighted extremal functions: {T}he {S}iciak-{Z}akharyuta theorem}, Complex
  Anal. Synerg., 10 (2024), p.~12.

\bibitem{MagSigSigSno:2023}
{\sc B.~S. Magn{\'u}sson, {\'A}.~E. Sigurðard{\'o}ttir, R.~Sigurðsson, and
  B.~Snorrason}, {\em Polynomials with exponents in compact convex sets and
  associated weighted extremal functions - {F}undamental results}, Ann. Polon.
  Math., 133 (2024), pp.~37--70.

\bibitem{Per:2024}
{\sc M.~Perera}, {\em Further regularity results in {$P$}-extremal setting}, J.
  Anal., 32 (2024), pp.~3297--3305.

\bibitem{Ransford:1995}
{\sc T.~Ransford}, {\em Potential theory in the complex plane}, vol.~28 of
  London Mathematical Society Student Texts, Cambridge University Press,
  Cambridge, 1995.

\bibitem{Rockafellar:1970}
{\sc R.~T. Rockafellar}, {\em Convex analysis}, vol.~No. 28 of Princeton
  Mathematical Series, Princeton University Press, Princeton, NJ, 1970.

\bibitem{Sic:1981}
{\sc J.~Siciak}, {\em {Extremal plurisubharmonic functions in {${\bf
  C}\sp{n}$}}}, Ann. Polon. Math., 39 (1981), pp.~175--211.

\bibitem{Sic:1997b}
\leavevmode\vrule height 2pt depth -1.6pt width 23pt, {\em Wiener's type
  sufficient conditions in {$\bold C^N$}}, Univ. Iagel. Acta Math.,  (1997),
  pp.~47--74.

\bibitem{Sig:1991}
{\sc R.~Sigurdsson}, {\em Convolution equations in domains of {${\bf C}^n$}},
  Ark. Mat., 29 (1991), pp.~285--305.

\bibitem{Sno:2024a}
{\sc B.~Snorrason}, {\em Polynomials with exponents in compact convex sets and
  associated weighted extremal functions -- {G}eneralized product property},
  Math. Scand., 130 (2024), pp.~538--554.

\end{thebibliography}

\smallskip\noindent
Science Institute, University of Iceland,
IS-107 Reykjavík,  ICELAND. 

\smallskip\noindent
bergur@hi.is.

\end{document}